\documentclass[12pt,a4paper,reqno]{amsart}\usepackage{amsfonts,amsthm,latexsym,amsmath,amssymb,amscd,amsmath, epsf}
\providecommand{\keywords}[1]{\textbf{\textit{Keywords:}} #1}

\usepackage[utf8]{inputenc}
\usepackage{latexsym,amssymb,amsthm}
\usepackage{url,graphics} 
\usepackage[dvips]{graphicx,epsfig} 
\usepackage{relsize}
\usepackage{multicol}

\newcommand{\Dim}{\textrm{dim}}

\newcommand{\Span}{\textrm{span}}

\newcommand {\bC} {\mathbb {C}}
\newcommand {\bR} {\mathbb {R}}

\newcommand {\bQ} {\mathbb {Q}}
 
\newcommand {\bK} {\mathcal{K}}

\newcommand{\B} {\mathcal{B}}
\newcommand{\C} {\mathcal{C}}
\newcommand{\D} {\mathcal{D}}

\newcommand{\HH} {{\mathcal{H}}}
\newcommand{\I} {\mathcal{I}}

\newcommand{\bbK} {\mathcal{K}}

\newcommand{\tX}{\widetilde{X}}

\newtheorem{theorem}{Theorem}
\newtheorem{proposition}[theorem]{Proposition}
\newtheorem{lemma}[theorem]{Lemma} \newtheorem{definition}{Definition}
\newtheorem{notation}[definition]{Notation}
\newtheorem{corollary}{Corollary} \newtheorem{example}{Example}

\newtheorem{remark}{Remark} 

\newtheorem{question}{Question}

\newtheorem{problem}[question]{Problem}

\usepackage{lipsum}
\makeatletter
\DeclareRobustCommand{\cev}[1]{%
  \mathpalette\do@cev{#1}%
}
\newcommand{\do@cev}[2]{%
  \fix@cev{#1}{+}%
  \reflectbox{$\m@th#1\vec{\reflectbox{$\fix@cev{#1}{-}\m@th#1#2\fix@cev{#1}{+}$}}$}%
  \fix@cev{#1}{-}%
}
\newcommand{\fix@cev}[2]{%
  \ifx#1\displaystyle
    \mkern#23mu
  \else
    \ifx#1\textstyle
      \mkern#23mu
    \else
      \ifx#1\scriptstyle
        \mkern#22mu
      \else
        \mkern#22mu
      \fi
    \fi
  \fi
}

\begin{document}
          \numberwithin{equation}{section}

          \title[On  $Q$-deformations of Postnikov-Shapiro algebras]{On  $Q$-deformations of Postnikov-Shapiro algebras}
\author[A.N. Kirillov]{Anatol N. Kirillov }
\address{Research Institute for Mathematical Sciences
Kyoto University,
 606-8502, Kyoto, 
Japan;
The Kavli Institute for the Physics and Mathematics of 
the Universe, 277-8583, Kashiwa, Japan
}
\email{kirillov@kurims.kyoto-u.ac.jp}

\author[G. Nenashev]{Gleb Nenashev}
\address{ Department of Mathematics,
   Stockholm University,
   S-10691, Stockholm, Sweden}
\email{nenashev@math.su.se}

\begin{abstract}For any given loopless graph $G$,  we introduce $Q$ - deformations of its Postnikov-Shapiro algebras counting spanning trees, counting spanning forests and $Q$ - deformations of internal algebra of $G$. We determine the total dimension of the algebras; our proof also gives a new proof of the formula for the total dimensions of the usual Postnikov-Shapiro algebras. Furthermore, we construct "square-free" definition of usual internal algebra of~$G$.
\end{abstract}


\thanks{The paper was presented at FPSAC'17 as a talk. Extended abstract was published in {Séminaire Lotharingien de Combinatoire, 78B.55, FPSAC (2017) 12 pp.}}

\maketitle

\section{Introduction and main results}

The Postnikov-Shapiro algebras (PS-algebras for short) have been introduced 
and studied in~\cite{PS}. There are a few generalizations of that algebras: in~\cite{AP} and~\cite{HR}, under the name {\it zonotopal algebras}, a genrralization of PS-algebras  algebra was introduced for (real) arrangements.  In fact, 
this topic has its origin in earlier papers~\cite{SS} and~\cite{PSS}, which 
were motivated by the following problem had  been posed by V.\,Arnold in~\cite{Ar}:  

 Describe algebra $ {\mathcal{C}_n}$  generated by the curvature forms of tautological Hermitian linear bundles over the type $A$ complete flag variety ${\mathcal{F}}l_n$.

Surprisingly enough, it was observed and conjectured in \cite{SS}, that 
$\dim_{Q} {\mathcal{C}}_n= {\mathcal{F}}_n$, where ${\mathcal{F}}_n$ denotes 
the number of spanning forests of the complete graph $K_n$  on $n$ labeled vertices.   This Conjecture has been proved in \cite{PSS}, and became a starting 
point for a wide variety of generalizations, including discovery of PS-algebras.

The PS-algebras have a number of interesting properties, including an explicit formula for their Hilbert polynomials. Also these algebra are related to Orlik-Terao~algebras~\cite{OTe}, for more details, see for example~\cite{B}.

 In our paper we will use the following basic notation:
\begin{notation}
{\rm We fix a field of zero characteristic $\bbK$ (for example $\bC$ or~$\bR$).

We will work only with graphs without loops, but possibly with multiple edges.  We denote by 
$E(G)$ and $V(G)$ the set of edges and vertices of $G$ respectively. The  cardinalities of  $E(G)$ and $V(G)$ are denoted by $e(G)$ and $v(G)$ resp. The number of connected components of $G$ is denoted by $c(G)$. 

 We denote the set $\{1,2,\ldots,(a-1),a\}$ by $[a]$.
}
\end{notation}

The following  algebra $\C_G$  (counting spanning forests)  associated to an arbitrary vertex-labeled graph $G$ was introduced in  \cite {PS}.
 Let $G$ be a graph without loops on the vertex set $[n].$ 
 Let  $\Phi_G$ be the graded commutative algebra over $\bbK$ 
generated by the variables $\phi_e, e \in G$, with the defining relations:
$$(\phi_e)^2 = 0, \quad \text {for every edge}\; e\in G.$$
Let $\C_G$ be the subalgebra of  $\Phi_G$ generated by the elements
$$X_i =\sum_{e\in G} c_{i,e} \phi_e,$$
for  $i \in  [n], $ where  
\begin{equation}\label{eq:def_cie}
c_{i,e}=\begin{cases} \;\;\;1\quad \text{if}\; e=(i,j),\ i<j;\\
                                            -1\quad\text{if}\; e=(i,j),\ i>j;\\
                                             \;\;\;0\quad \text{otherwise}.
\end{cases}
\end{equation}
Observe that we
 assume that $\C_G$ contains $1$.

\begin{example} Consider graph $G$ as on picture. 

 {\begin{figure}[htb!]
\centering
\includegraphics[scale=0.8]{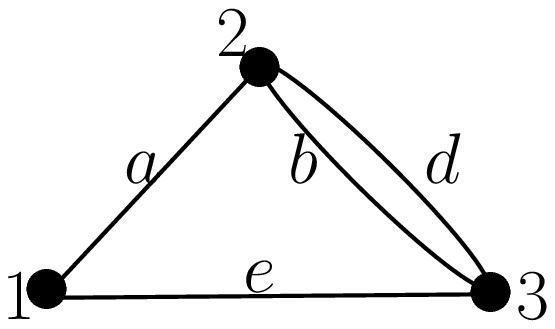}
\end{figure}
}

$$X_1=\phi_a+\phi_e;$$
$$X_2=-\phi_a+\phi_b+\phi_d;$$
$$X_3=-\phi_b-\phi_d-\phi_e.$$

The algebra $\C_G$ is generated by $X_1,X_2,X_3$, namely $$\C_G:=span\{1,\ X_1,\ X_2,\ X_1^2,\ X_1X_2,\ X_2^2,\  X_1^2X_2,\ X_1X_2^2,\ X_2^3,\ X_1^2X_2^2\}.$$
\end{example}

\medskip

Let us describe all relations between $X_i$.
{\rm For given a graph $G$,
consider the ideal $J_{G}^k$ in the ring $\bbK[x_1,\cdots,x_n]$ generated by
$$p_I^{(k)}=\left(\sum_{i\in I} x_i\right)^{ d_I+k},$$
where $I$ ranges over all nonempty subsets of vertices, and $d_I$ is the total number of edges between vertices in $I$ and vertices outside~$I$, i.e., belonging to $V(G)\setminus I$.  Define the algebra   ${\B}_G, \B_G^T, \B_G^{In}$ as the quotient $\bbK[x_1,\dots,x_n]/ J_{G}^k,$ for $k=1,0,-1$ resp.}

\begin{theorem}[cf.~\cite{PSS, PS}]
\label{thm:PS}

 For any graph $G$, the algebras  ${\B}_{G}$ and ${\C}_{G}$ are isomorphic,
 their total dimension over $\bbK$ is equal to the number of spanning forests in $G$.

Moreover, the dimension of the $k$-th graded component of these algebras  equals
the number of spanning forests $F$ of $G$ with external activity $e(G)-e(F)-k$.
\end{theorem}

Algebras ${\C}_G={\B}_G$ is called PS algebra counting spanning forests or external algebra.
 We will discuss cases $\B^T_G$ and $\B^{In}_G$ in \S~5, these algebras are called algebra counting spanning trees (central algebra) and internal algebra.

In particular, the second part of Theorem~\ref{thm:PS} implies that the Hilbert polynomial of ${\C}_{G}$ is a specialization of the Tutte polynomial 
of~$G$.   

\begin{corollary}
Given a graph $G$, the Hilbert polynomial $\HH_{C_G}(t)$ of the algebra $\C_G$ is given by
$$\HH  _{\C_G}(t)=T_G\left(1+t,\frac{1}{t}\right)\cdot t^{e(G)-v(G)+c(G)}.$$ 
\end{corollary}
In the recent paper~\cite{Ne} the second author found the following important property of these algebras.
\begin{theorem}[cf.~\cite{Ne}] Given two graphs $G_1$ and $G_2,$ the algebras $\C_{G_1}$ and $\C_{G_2}$ are isomorphic  if and only if the graphical matroids of $G_1$ and $G_2$ coincide. (The  isomorphism can be thought of as either graded or non-graded, the statement holds in both cases.) 
\end{theorem}

Furthermore, the paper~\cite{NS} contains a "K-theoretic" filtered structure of these algebras, which  distinguishes graphs (see definition inside there).

\bigskip

The main object of study of the present paper  is a family of $Q$-deformations of  $\C_G$ which we define as follows. For a graph $G$  and a set  of parameters $Q=\{q_e\in \bbK: \ e\in E(G)\}$, define 
 $\Phi_{G,Q}$ as the commutative algebra generated by the variables $\{u_e : \ e\in E(G)\}$ satisfying 
 $$u_e^2=q_eu_e,\; \text{for every edge\; } e\in G.$$

Let $V(G)=[n]$ be the vertex set  of a graph $G.$ Define the \emph{$Q$-deformation} $\Psi_{G,Q}$ of $\C_G$ as the filtered subalgebra of $\Phi_{G,Q}$ generated by the elements:
$$X_i=\sum_{e: \ i\in e} c_{i,e}  u_e, \ i\in[n],$$
where $c_{i,e}$ are the same as in~\eqref{eq:def_cie}. The filtered structure on $\Psi_{G,Q}$ is induced by the elements $X_i,\; i\in[n]$. More concrete, the filtered structure is an increasing sequence   $$\bbK= F_0\subset F_1\subset F_2\ldots \subset F_m=\Psi_{G,Q}$$ of subspaces of $\Psi_{G,Q}$, where $F_{k}$ is the linear span of all monomials $X_1^{\alpha_1}X_2^{\alpha_2}\cdots X_n^{\alpha_n}$ such that $\alpha_1+\ldots+\alpha_n\leq k$. Note that algebra $\Phi_{G,Q}$ has a finite dimension, then $\Psi_{G,Q}$ has a finite dimension, which gives that the increasing sequence of subspaces is finite.  
The Hilbert polynomial of a filtered algebra is the Hilbert polynomial of the associated graded algebra, it has the following formula
 $$\HH(t)=1+\sum_{i=1} (dim(F_i)-dim(F_{i-1}))t^{In}.$$
 
\smallskip
 In case when all parameters coincide, i.e., $q_e=q$, $\forall e\in G,$  we denote the corresponding algebras by  $\Psi_{G,q}$ and $\Phi_{G,q}$ resp. We refer to $\Psi_{G,q}$ as the {\it Hecke deformation} of $\C_G$.

 \begin{remark}{\rm (i)}
By definition, the algebra $\Psi_{G,0}$ coincides with  $\C_G$. 

{\rm (ii)}
If we change the signs of $q_e,\ e\in E'$ for some subset $E'\subseteq E$ of edges, we obtain an isomorphic algebra. 

{\rm(iii)} It is possible to write relations such as $u_e^2=\beta_e$ or $u_e^2=q_e u_e+\beta_e$  where ${\beta_e\in \bbK}$. 
But in the case of algebras counting spanning trees we need relations without constant terms, see \S~\ref{final}.
\end{remark}

\begin{example}
\emph{(i)} Let $G$ be a graph with two vertices, a pair of (multiple) edges $a$, $b$. Consider the Hecke deformation of its $\C_G$, i.e., satisfying $q_{a}=q_{b}=q$.

\smallskip 
The generators are $X_1=a+b,\; X_2=-(a+b)=-X_1.$ One can easily check that the filtered structure is given by
\begin{itemize}
\item $F_0=\bbK={<}1{>}$;
\item $F_1={<}1,\ a+b{>}$;
\item $F_2={<}1,\ a+b,\ ab{>}$.
\end{itemize}
The Hilbert polynomial $\HH  (t)$ of $\Psi_{G,q}$ is given by 
$$\HH  (t)=1+t+t^2.$$
 The defining relation for $X_1$  is given by
 $$X_1(X_1-q)(X_1-2q)=0.$$

\noindent
\emph{(ii)} For the same graph as before, consider the case when $Q=\{q_a,q_b\},\;q_{a}^2\neq q_{b}^2$.

The generators are the same: $X_1=a+b$, $X_2=-(a+b)=-X_1$. 
Since 
\begin{multline*}X_1^3=q_a^2a+q_b^2b+3(q_a+q_b)ab=\frac{3(q_a+q_b)}{2}X_1^2-\frac{q_a^2+3q_b^2}{2}a-\frac{3q_a^2+q_b^2}{2}b\\
=\frac{3(q_a+q_b)}{2}X_1^2-\frac{3q_a^2+q_b^2}{2}X_1+(q_a^2-q_b^2)a,
\end{multline*}
\begin{itemize}
\item $F_0=\bbK={<}1{>}$;
\item $F_1={<}1,\ a+b{>}$;
\item $F_2={<}1,\ a+b,\ q_aa+q_bb+2ab{>}$;
\item $F_3={<}1,\ a,\ b,\ ab{>}$.
  \end{itemize}The Hilbert polynomial $\HH  (t)$ of $\Psi_{G,Q}$ is given by 
$$\HH  (t)=1+t+t^2+t^3.$$ 
Observe that in this case the algebra $\Psi_{G,Q}$ coincides with  the whole $\Phi_{G,Q}$ as a linear space, but has a different filtration.
 The defining relation for $X_1$  is given by
 $$X_1(X_1-q_a)(X_1-q_b)(X_1-q_a-q_b)=0.$$
\end{example}

\medskip

The first result of the present paper is about Hecke deformations. 
\begin{theorem}
\label{thm:same}
For any loopless graph $G$, filtrations of  its Hecke deformation $\Psi_{G,q}$  induced by $X_i$ and induced by the algebra $\Phi_{G,q}$  coincide. Furthermore, the Hilbert polynomial $\HH  _{\Psi_{G,q}}(t)$ of this filtration is given by 
$$\HH  _{\Psi_{G,q}}(t)=T_G\left(1+t,\frac{1}{t}\right)\cdot t^{e(G)-v(G)+c(G)},$$
i.e., it coincides with that of $\C_G$.
\end{theorem}

\medskip
The latter result implies that cases when not all   $q_e$ are equal are  more interesting than the case of the Hecke deformation. 
We will work with weighted graphs, i.e. when each edge $e$ has non-zero $q_e\in \bbK$, and will simply denote the algebra  for a weighted graph $G$ by $\Psi_G$.

\begin{definition}For a loopless weighted graph $G$ on $n$ vertices and an orientation $\vec{G},$ define the score vector $D_{\vec{G}}^+\in\bbK^n$ as follows
$$\Bigg(\sum_{\substack{e\in E:\\ end(\vec{e})=1}}q_e, \sum_{\substack{e\in E:\\ end(\vec{e})=2}}q_e,\ldots, \sum_{\substack{e\in E:\\ end(\vec{e})=n}}q_e \Bigg),$$
where $end(\vec{e})$ is the final vertex of oriented edge $\vec{e}$.
\end{definition}
\begin{theorem}
\label{thm:different}
 For any loopless weighted graph $G$, the dimension of the algebra $\Psi_G$ is equal to the number of distinct score vectors, i.e.
$$dim(\Psi_G)=\#\{D\in\bbK^n:\ \exists \vec{G} \text{\ such that }D=D_{\vec{G}}^+ \}.$$
\end{theorem}

\smallskip

As a consequence of Theorems~\ref{thm:same} and~\ref{thm:different}, we obtain the following known property.  (See bijective proofs in~\cite{KW} and~\cite{Be}.)
\begin{proposition}\label{cor:forests=orient}
For any graph $G$, the number of its spanning forests is equal to the number of distinct vectors of incoming degrees corresponding to its orientations.
\end{proposition}
Our proof of Theorem~\ref{thm:different} is very simple and it gives a new proof about total dimension of an original algebra.
Unlucky, our proof works only for weighted graphs (nonzeroes parameters). A zero parameter does not play role in score vectors, so we even do not have a conjecture.
\begin{problem}
What is the dimension of $\Psi_{G,Q}$ in case when some of $q_e$ are non-zeroes and few  are zeroes?
\end{problem}

\bigskip 

The structure of the paper is as follows. 
 In \S~\ref{sec:hecke} we prove Theorem~\ref{thm:same} and discuss Hecke deformations. In \S~\ref{basis} we describe the basis  of $Q$-deformations and  present a  proof of Theorem~\ref{thm:different}. In \S~\ref{experiments} we consider "generic" cases and provide  examples of Hilbert polynomials. In \S~\ref{final} we present 
 $Q$-deformations of the Postnikov-Shapiro algebra which counts spanning trees and of the internal algebra, we also present ``square free'' definition of the internal algebra.

\bigskip

\section{Hecke deformations}\label{sec:hecke}
\begin{proof}[Proof of Theorem~\ref{thm:same}]
To settle  this theorem, we need to show that if an element $y\in\Psi_{G,Q}$ has degree $d,$ then it has the same degree in $\Phi_{G,Q}$.

Assume the opposite; then there exists an element $y=f(X_1,\ldots,X_n)$, where $f$  is a polynomial of degree $d$, but $y$ has degree less than $d$  in its representation in terms of the edges $u_e,$ $e\in G$.

Rewrite $f$ as $f=f_d+f_{<d}$, where $f_d$ is a homogeneous polynomial of degree $d$ and  $\deg f_{<d} <d$.

Let $\widehat{X}_1,\ldots,\widehat{X}_n$ be the elements in the algebra ${\mathcal{C}}_G=\Psi_{G,0}$ corresponding to the vertices. 
 We conclude that $f_d(\widehat{X}_1,\ldots,\widehat{X}_n)$ should vanish. Indeed, otherwise  $\deg f_d({X}_1,\ldots,X_n)=d$  in $\Phi_{G,Q}$ and $\deg f_{<d}({X}_1,\ldots,X_n)<d$ which implies that  $\deg f({X}_1,\ldots,X_n)=d$  in $\Phi_{G,Q}$.
 
By Theorem~\ref{thm:PS}, we know all  the relations between $\{\widehat{X}_1,\ldots,\widehat{X}_n\}.$ Namely, they are of the  form $(\sum_{i\in I}\widehat{X}_i)^{d_I+1}$, where $I$ is an arbitrary subset of vertices and $d_I$ is the number of edges between $I$ and its complement $V(G)\setminus I$.

Using this, we obtain
$$f_d(x_1,\ldots,x_n)=\sum_{\substack{I\subseteq V(G):\\ d_I\leq d-1}} r_{I}(x_1,\ldots,x_n)\cdot \left(\sum_{i\in I}x_i\right)^{d_I+1},$$
where $r_I$ is a homogeneous polynomial of degree $d-d_I-1$.
However, it is possible to rewrite   $\left(\sum_{i\in I}X_i\right)^{d_I+1}$ as an element of a smaller degree in terms of $\{X_i,\ i\in I\}$. Hence, there is polynomial $g$ of degree less than $d$ such that $y=g(X_1,\ldots,X_n)$.

\smallskip

The second part follows from the first one. It is enough to consider graded lexicographic orders of monomials in $\{u_e,\ e\in G\}$ and $\{\phi_e,\ e\in G\}$. For these orders, we have a natural bijection between the Gr{\"o}bner bases of $\Psi_{G,q}$  and of ${\mathcal{C}}_G$. Hence, their Hilbert polynomials coincide.
\end{proof}

\medskip

Proposition~\ref{cor:forests=orient} shows that the dimension of a Hecke deformation is equal to the number of lattice points of the zonotope $Z\in \bR^n$, which is the Minkowski sum of edges, i.e,
$$Z_G:=\bigoplus_{e\in G}I_e,$$
where, for edge $e=(i,j)$, $I_e$ is the segment between points $(\underbrace{0,\dots, 0}_{i-1},1,0,\ldots,0)$  and $(\underbrace{0,\dots, 0}_{j-1},1,0,\ldots,0)$. Note that sum of all coordinates is $|E|$, i.e., corresponding zonotope belongs to subspace of dimension $n-1$, see example.
 {\begin{figure}[htb!]
\centering
\includegraphics[scale=0.9]{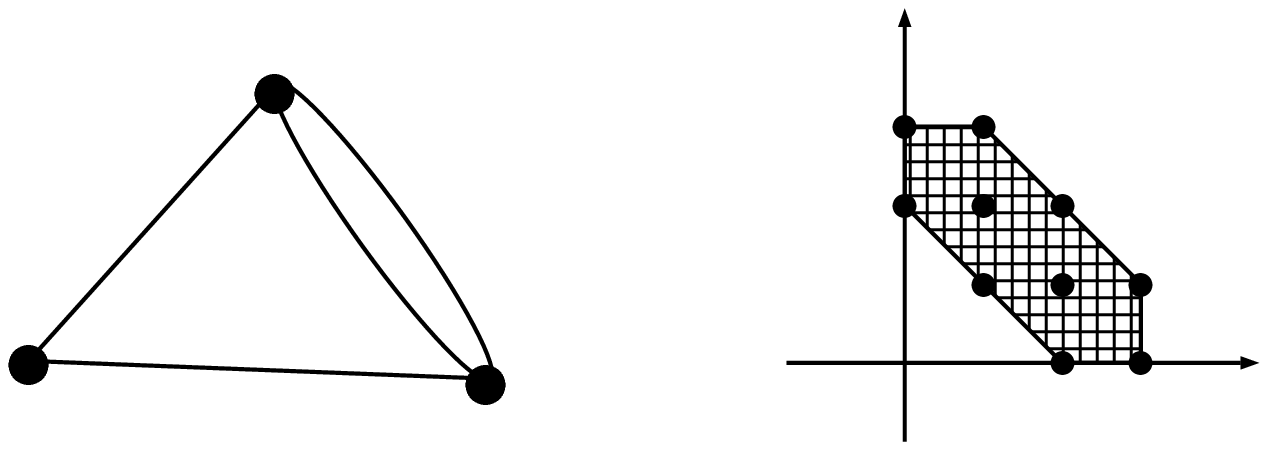}
\end{figure}
}
In~\cite{HR} Holtz and Ron defined the zonotopal algebra for any lattice zonote, whose dimension is equal to the number of lattice points. 
By their definition PS-algebra ${\B}_{G}$ is the zonotopal algebera corresponding to $Z_G$. We think that Hecke deformations should be extended on a case of zonotopal algebras.
\begin{problem}
Define Hecke deformations of zonotopal algebras.
\end{problem}

Since there is no definition of zonotopal algebras in terms of square-free algebras, we should work with quotient algebras. 
In case of Hecke deformations of PS-algebras Proposition~\ref{prop:ann} from  \S~\ref{basis} gives all defining relations between elements $X_i,\ i\in[n].$ 
\begin{theorem}
\label{thm:quo-Hecke}
Let $G$ be a graph and $q\in \bbK$ ($q_e=q,\ \forall e\in G$). Then all defining relations between $X_i,\ i\in [n]$ are given by
$$\prod_{k=-\vec{d}_I}^{\cev{d}_I}\left(\sum_{i\in I} X_i -qk\right)=0,$$
where $I$ is any subset of vertices and $\vec{d}_I$ (resp. $\cev{d}_I$) is the number of edges $e=(i,j)\in G:\ i\in I,\ j\notin I$ and $i>j$ (resp. $i<j$).  
\end{theorem}
\begin{proof}
By proposition~\ref{prop:ann} there are relations between $X_i, i\in [n]$ of type
$$\prod_{k=0}^{D_I}\left(\sum_{i\in I} X_i -qk\right)=0.$$
Consider the commutative algebra $\Psi_{G,q}'$ generated by $x_i$ and with relations $$\prod_{k=0}^{D_I}\left(\sum_{i\in I} x_i -qk\right)=0,$$
for any subset $I\subset [n]$. 

We know that $$\Dim(\Psi_{G,q}')\geq \Dim(\Psi_{G,q}) = \Dim(\Psi_{G,0}) = \Dim(\B_G),$$
in other hand it is clear that $\Dim(\B_g)\geq \Dim(\Psi_{G,q}').$ We obtain $$\Dim(\Psi_{G,q}')=\Dim(\Psi_{G,q}),$$
hence, $\Psi_{G,q}$ and  $\Psi_{G,q}'$ are isomorphic.
\end{proof}

\bigskip

\section{Basis of  $Q$-deformations}\label{basis} For the next proofs, we need to describe a basis of the algebra $\Phi_G$. For a subset $E'$ of the edges, we define $$\alpha_{E'}=\prod_{e\in E'}\frac{u_e}{q_e}.$$
Since $q_e\neq 0$ this basis is well defined. For an element $z=\sum_{E'} z_{E'}\alpha_{E'}\in \Phi_G$, we define the vector $\widetilde{z}=[\widetilde{z}_{E'}]_{E'\subseteq E}\in \bbK^{2^{e(G)}},$ where
$$\widetilde{z}_{E'}=\sum_{E''\subseteq E'}z_{E''}.$$
It is clear that from this vector we can reconstruct $z$, also it is easy to describe the product on these coordinates. Furthermore unit element $I$ is given by $I:=\widetilde{1}=[1]_{E'\subseteq E}.$
\begin{lemma}\label{lem:basis} Elements corresponding to $[0,\ldots,0,1,0,\ldots,0]$ form a linear basis of~$\Phi_G$. This basis has the following property: let $y,z\in \Phi_G,$ be elements of the algebra, then the sum of elements is the sum by coordinates
$$\widetilde{(y+z)}=\widetilde{y}+\widetilde{z},$$
and the product is the Hadamard product of coordinates
$$\widetilde{(yz)}=\widetilde{y}\circ \widetilde{z}.$$
\end{lemma}
\begin{proof}
The part about the summation is clear.

For any vector we can use M{\"o}bius inversion formula and get its element in the algebra $\Phi_{G,q}$. Since the dimension of the space of our vectors is $2^{e(G)}$, which is also the dimension of the algebra $\Phi_{G,q}$, elements corresponding to $[0,\ldots,0,1,0,\ldots,0]$ form a linear basis.

It is easy to check that for $E_1,E_2\subseteq E$, we have $$\alpha_{E_1}\alpha_{E_2}=\alpha_{E_1\cup E_2}.$$
Then we obtain 
$$(yz)_{E'}=\sum_{\substack{E_1,E_2:\\ E_1\cup E_2=E'}} y_{E_1}z_{E_2}.$$
After the change of coordinates, we get 
\begin{multline*} \widetilde{(yz)}_{E'}=\sum_{E''\subseteq E'}\sum_{\substack{E_1,E_2:\\ E_1\cup E_2=E''}} y_{E_1}z_{E_2}=\sum_{\substack{E_1,E_2:\\ E_1\cup E_2\subseteq E'}} y_{E_1}z_{E_2}\\
=
\left(\sum_{E_1\subseteq E'} y_{E_1}\right)\left(\sum_{E_2\subseteq E'} z_{E_2}\right)=\widetilde{y}_{E'}\widetilde{z}_{E'}.
\end{multline*}
Then our product in these coordinates coincides with the Hadamard product.
\end{proof}

Consider the following bijection between subsets of $E(G)$ and orientations of $G$.
For the subset $E'\subseteq E$ we define the following orientation: if $e\in E'$, then the orientation is from the biggest end to the smallest, otherwise the orientation is the opposite.

\begin{lemma}
\label{lem:Xi}
The element $X_i$ in coordinates is given by
$$\widetilde{X}_i
=\begin{bmatrix}\makebox[4em]{}\\
\begin{aligned}
 D_{\vec{G}}^+(i)
\end{aligned}
\\
\makebox[4em]{} \end{bmatrix}_{\vec{G}} - \left(\sum_{\substack{e\in E:\\
c_{i,e}=-1}} q_e\right) \cdot I,$$
where $D_{\vec{G}}^+(i)$ is $i$-th coordinate of a score vector $D_{\vec{G}}^+$.
\end{lemma}
\begin{proof}
Recall the definition $$X_i=\sum_{e: \ i\in e} c_{i,e}  u_e, \ i\in[n],$$
which gives $$\tX_{i,E'}=\sum_{e\in E': \ i\in e} c_{i,e}  q_e, \ i\in[n].$$ Note that
$$D_{E'}^+(i)=\left(\sum_{\substack{e\in E':\\
c_{i,e}=1}} q_e\right)+\left(\sum_{\substack{e\notin E':\\
c_{i,e}=-1}} q_e\right),$$
this equality with previous finish our proof of lemma.
\end{proof}  

\medskip

We use in the proof of Theorem~\ref{thm:different} the following elements $$\widetilde{A}_i:
=\begin{bmatrix}\makebox[4em]{}\\
\begin{aligned}
 D_{\vec{G}}^+(i)
\end{aligned}
\\
\makebox[4em]{} \end{bmatrix}_{\vec{G}}.$$
We need another technical lemma.
\begin{lemma}
\label{lem:dim}
For an element $R\in \Phi_G$, the dimension of the space generated by $R$ (i.e,  
$\Span{{<}1,R,R^2,\ldots{>}}$) is equal to the number of different coordinates of the vector $\widetilde{R}$.
\end{lemma}
\begin{proof}
Denote by $\mathcal M$ the set of all values of coordinates of $\widetilde{R}$.
Consider an annihilating polynomial of $R$
$$f(x):=\prod_{\beta\in {\mathcal M}} (x-\beta)\ \ \ \ f(R)=\prod_{\beta\in {\mathcal M}} (R-\beta)=0,$$
really it is annihilating polynomial, because by Lemma~\ref{lem:basis} we have
$$\widetilde{f(R)}=\prod_{\beta\in {\mathcal M}} (\tilde{R}-\beta\cdot I).$$
Check coordinate $E'\subseteq E$: $R_{E'}\in {\mathcal M}$, then there is factor of the product, which has zero $E'$-coordinate. Hence, the product has zero $E'$-coordinate. Then all coordinates are zeroes, i.e., $f$ is an annihilating polynomial of $R$.

Let $g(x)$ be the minimal unitary annihilating polynomial  of $R$, it is clear that
$$\deg(g)=\dim(\Span{{<}1,R,R^2,\ldots{>}}).$$ 
We have $g|f$, if $\dim(\Span{{<}1,R,R^2,\ldots{>}})<|{\mathcal M}|,$ then there is $\alpha \in {\mathcal M}$ such that $g|\frac{f}{(x-\alpha)}$, hence,
$$\prod_{\beta\in {\mathcal M}\setminus \{\alpha\}} (R-\beta)=0.$$
Consider coordinate $E'\subseteq E$ such that $R_{E'}=\alpha$. We have $(R-\beta)_{E'}\neq,\ \beta\in {\mathcal M}\setminus \{\alpha\}$, hence, the previous product has no zero $E'$-coordinate, which impossible. We obtain $g=f$, which finishes our proof.
\end{proof}

\medskip 

Now we can prove Theorem~\ref{thm:different}.
\begin{proof}[Proof of Theorem~\ref{thm:different}]
By Lemma~\ref{lem:Xi} we can change the set of generators $X_i,\ i\in V(G)$ to the set $A_i,\ i\in V(G)$.
If two orientations have the same score vector, then the corresponding coordinates in $\I$ and in $\widetilde{A}_i,\ i\in V(G)$ coincide. Using Lemma~\ref{lem:basis}, we get that they coincide for any element from algebra $\Psi_G$, hence,  $$dim(\Psi_G)\leq \#\{D\in\bK^n:\ \exists \vec{G} \text{\ such that }D=D_{\vec{G}}^+ \}.$$
For converse, we consider an element $$R=r_0 + r_1 A_1+\ldots +r_nA_n, $$
where $r_i\in \bQ$ and are generic. 

The coordinates $\widetilde{R}$ are non-zeroes and, for two orientations, they coincide if and only if their score vectors coincide. Then, by Lemma~\ref{lem:dim} the dimension of the subalgebra generated by $R$ is equal to number of different score vectors. Since $R$ belongs to $\Psi_G$, we obtain 
$$dim(\Psi_G)\geq \#\{D\in\bK^n:\ \exists \vec{G} \text{\ such that }D=D_{\vec{G}}^+ \},$$
which with the upper bound gives equality.
\end{proof}

\bigskip

Using Lemma~\ref{lem:dim}
 we can calculate the minimal annihilating polynomial for any linear combination of vertices.
\begin{proposition}
\label{prop:ann}
Given  weighted graph $G$. For an element $X\cdot t=X_1t_1+\ldots+X_nt_n,\ t\in \bK^n$  the minimal annihilating polynomial of it is given by
$$\prod_{s\in \D_I}\left(X\cdot t-s+z\right)=0,$$
where $${\D}_I=\{D_{\vec{G}}^+\cdot t:\  \ \vec{G}\}\ \ \ {\textrm and} \ \ \ z=\sum_{\substack{i,\ e:\\
c_{i,e}=-1}} q_e t_i.$$
\end{proposition}
\begin{proof} We have everything inside Lemma~\ref{lem:dim}. We should just find the set of all values of coordinates of $\tilde{X}$, which we know by Lemma~\ref{lem:Xi}.
\end{proof}

In case of Hecke deformations it gives all defining relations between $X_i,\ i\in V(G)$, see Theorem~\ref{thm:quo-Hecke}.
\begin{problem}
Find all relations between $X_i,\ i\in V(G)$. In other words, define  $\Psi_{G,Q}$ 
as a quotient algebra of the polynomial ring.
\end{problem}

\bigskip

\section{Case $E=E_1\sqcup\ldots\sqcup E_k$ and generic $q_1,\ldots,q_k\in \bK$}
\label{experiments}

We can not describe the Hilbert polynomial of $\Psi_{G,Q}$.  We suggest to start from the following type of algebras: when different parameters are in a generic position. In this case we know the total dimension in terms of forests.
\begin{theorem}
\label{thm:split}
Let $G$ be a graph,  given a partition $E=E_1\sqcup\ldots\sqcup E_k$ of edges and generic $q_1,\ldots,q_k\in \bK$ ($q_e=q_i$, for $e\in E_i$). Then the dimension of the algebra $\Psi_{G,Q}$ equals the number $k$-tuples of spanning forests such that $F_i\subseteq E_i$. In other words, 
$$\dim (\Psi_{G,Q})=\prod_{i=1}^k\#\{F\subseteq E_i\;\vert\quad  F\;\text{ is a forest}\}.$$
\end{theorem}

\begin{problem}
\label{ques:split}
What is the Hilbert polynomial $HS_{\Psi_{G,Q}}$ in case $E=E_1\sqcup\ldots\sqcup E_k$ and generic $q_1,\ldots,q_k\in \bK$?

It seems that it is impossible to reconstruct the Hilbert polynomial  from the Tutte polynomial.
For example, let $G$ be the graph on two vertices with $k$ multiply edges, then its Tutte polynomial is given by $$T_G(x,y)=x+y+\ldots+y^{k-1},$$
and the Hilbert polynomial, when each edge has a self generic parameter is
$$HS_{\Psi_{G,Q}}=1+t+\ldots+t^{2^k-1}.$$
In each case it is not a specialization of the Tutte polynomial. 
\end{problem}

Here we present the Hilbert polynomial of algebras for complete graphs. Our tables correspond to algebras (1) with the same parameter; (2) with the same parameters except for one edge and (3) where all parameters are generic. By Theorem~\ref{thm:split} we know their total dimensions, in the first case we also know the Hilbert polynomial.

\subsection{Hilbert polynomials of $\C_{K_n}$ and $\Psi_{K_n,q}$}

\begin{center}
    \begin{tabular}{ | l | l | l | l | l | l | l |  l | l | l | l | l|p{5cm} |}
    \hline
    Graph $\backslash \HH  (t)$  & $0$ & $1$  & $2$ & $3$  & $4$ & $5$ & $6$& $7$& $8$ & $9$& $10$ \\ \hline
    
    $K_2$& $1$ & $1$ & &&&&&&&& \\ \hline
    
   $K_3$& $1$ & $2$ & $3$ & $1$ & &&&&&& \\ \hline
   
   $K_4$& $1$ & $3$ & $6$ & $10$ & $11$ & $6$ & $1$ &&&&\\ \hline 
     
     $K_5$& $1$ & $4$ & $10$ & $20$ & $35$ & $51$ & $64$ & $60$ & $35$ & $10$ & $1$\\ \hline 
    \end{tabular}
\end{center}

\subsection{Hilbert polynomials of $\Psi_{K_n,Q}$, when $E_1=E(K_n)\backslash \{e\}$ and $E_2=\{e\}$}

\begin{center}
    \begin{tabular}{ | l | l | l | l | l | l | l |  l | l | l | l | l|  l|p{5cm} |}
    \hline
    Graph $\backslash  \HH  (t)$  & $0$ & $1$  & $2$ & $3$  & $4$ & $5$ & $6$& $7$& $8$ & $9$& $10$\\ \hline
    
    $K_2$& $1$ & $1$ & &&&&&&&&\\ \hline 
    
   $K_3$& $1$ & $2$ & $3$ & $2$ & &&&&&& \\ \hline 
   
   $K_4$& $1$ & $3$ & $6$ & $10$ & $13$ & $11$ & $4$&&&& \\ \hline 
    $K_5$& $1$ & $4$ & $10$ & $20$ & $35$ & $53$ & $72$ & $83$ & $72$ & $38$ & $8$\\ \hline 
     
    \end{tabular}
\end{center}

\subsection{Hilbert polynomials of $\Psi_{K_n,Q}$, when $Q$ is generic}
\label{table:generic}

\begin{center}
    \begin{tabular}{ | l | l | l | l | l | l | l |  l | l | l | l | l|  l|p{5cm} |}
    \hline
    Graph $\backslash \HH  (t)$  & $0$ & $1$  & $2$ & $3$  & $4$ & $5$ & $6$& $7$& $8$ & $9$ & $10$ & $11$\\ \hline
    
    $K_2$& $1$ & $1$ & &&&&&&&&&\\ \hline 
    
   $K_3$& $1$ & $2$ & $3$ & $2$ &&&&&&&& \\ \hline 
   
   $K_4$& $1$ & $3$ & $6$ & $10$ & $15$ & $19$ & $10$ &&&&&\\ \hline 
    $K_5$& $1$ & $4$ & $10$ & $20$ & $35$ & $56$ & $84$ & $120$ & $165$ & $220$ & $217$ & $92$ \\ \hline  
    \end{tabular}
\end{center}
Note that in last case for $K_5$, the $11$-th graded component is not empty, because otherwise the total dimension at most $1+4+10+..+220+286=1001$, but by Theorem~\ref{thm:different} the total dimension is $2^{\binom{5}{2}}=1024$.

\bigskip

\section{Deformations of Postnikov-Shapiro algebras counting spanning trees and internal algebras}
\label{final}
In this section we consider other types of Postnikov-Shapiro algebras, namely $\B_G^T$ counting spanning trees and internal algebra $\B_G^{In}$ defined in~\cite{AP,HR}. Their Hilbert series are known.
\begin{theorem}[{\bf T} cf.~\cite{PS}; {\bf In} cf.~\cite{AP,HR}]
For graph $G$, the Hilbert polynomials $\HH_{\B_{G,q}^*}(t)$ of this filtration are given by 
$$\HH  _{\B_G^T}(t)=T_G\left(1,\frac{1}{t}\right)\cdot t^{e(G)-v(G)+c(G)},$$
$$\HH  _{\B_G^{In}}(t)=T_G\left(0,\frac{1}{t}\right)\cdot t^{e(G)-v(G)+c(G)}.$$
\end{theorem}

At first we define algebras $\Phi_{G,Q}^T$ and $\Phi_{G,Q}^{In}$. To shorten the notations  define the following polynomial for a subset of vertices $I\subseteq V(G)$
$$f_I^*=\prod_{e\in E(I,\bar{I})} u_e^*\prod_{e\in E(\bar{I},I)} (u_e^*-q_e),$$
where $E(I,\bar{I})=\left\{\substack{e=(i,j)\ i\in I, j\notin I:\\ c_{i,e}=1}\right\}$ and $E(\bar{I},I)=\left\{\substack{e=(i,j)\ i\in I, j\notin I:\\ c_{i,e}=-1}\right\}$. ($*$ is either $T$ or ${In}$)

For a connected graph $G$  and a set  of parameters $Q=\{q_e\in \bK: \ e\in E(G)\}$, fix a vertex $g$, define 
 $\Phi_{G,Q}^T$,  as the commutative algebra generated by the variables $\{u_e^T : \ e\in E(G)\}$ and  satisfying 
 $$(u_e^T)^2=q_e u_e^T,\; \text{for every edge\; } e\in G;$$
  $$f_I^T=0,\; \text{for every subset $g\in I\subseteq V(G)$}.$$
In case when all $q_e\neq 0$, define 
 $\Phi_{G,Q}^{In}$ as the commutative algebra generated by the variables $\{u_e^{In} : \ e\in E(G)\}$ satisfying 
 $$(u_e^{In})^2=q_e u_e^{In},\; \text{for every edge\; } e\in G;$$
 
$$ f_I^{In}=0,
 \text{for every subset $ I\subseteq V(G)$}.$$

Let $V(G)=[n]$ be the vertex set  of a graph $G.$ Define the algebra  $\Psi_{G,Q}^T$ and $\Psi_{G,Q}^{In}$ as a filtered subalgebras of $\Phi_{G,Q}^T$ and f $\Phi_{G,Q}^{In}$ generated by the elements:
$$X_i^T=\sum_{e: \ i\in e} c_{i,e}  u_e^T, \ i\in[n],$$
$$X_i^{In}=\sum_{e: \ i\in e} c_{i,e}  u_e^{In}, \ i\in[n],$$
where $c_{i,e}$ are the same as in~\eqref{eq:def_cie}.

\begin{remark}
Note that we can write one equation $f_I^{In}-f_{\bar I}^{In}=0$ instead of two equations $f_I^{In}=f_{\bar I}^{In}=0$.
Really, assume $c_{i,e}=1$  (if $=-1$ everything is similar), then we have 
$$u_e^{In}(f_I^{In}-f_{\bar I}^{In})=q_ef_I^{In},$$
and 
$$(u_e-q_e)^{In}(f_I^{In}-f_{\bar I}^{In})=q_ef_I^{In}.$$
Sometimes that can be useful, because we decrease the number of equations and lower the degrees of those equations. 
\end{remark}

\medskip
 In case when  all parameters coincide, i.e., $q_e=q$, $\forall e\in G,$  we denote the corresponding algebras by  $\Psi_{G,q}^T$ and $\Psi_{G,q}^{In}$ resp. The algebra $\Psi_{G,0}^T$ coincides with  $\C_G^T$, the dimension of $\C_G^T$ is equal to the number of spanning trees (see~\cite{PS}). We refer to $\Psi_{G,q}^T$ and $\Psi_{G,q}^{In}$ as the {\it Hecke deformation}.

\begin{theorem}[cf.~\cite{PS}]
\label{thm:PS-tree}

 For any graph $G$, the algebras  ${\B}_{G}^T$ and $\Psi_{G,0}^T={\C}^T_{G}$ are isomorphic,
 their total dimension over $\bK$ is equal to the number of spanning trees in $G$.

Moreover, the dimension of the $k$-th graded component of these algebras  equals
the number of spanning trees $T$ of $G$ with external activity $e(G)-e(T)-k$.
\end{theorem}

When some of $q_e=0$ the algebra $\Psi_{G,Q}^{In}$ is not correctly defined, for example, from the above definition $\Psi_{G,0}^{In}$  coincides with $\Psi_{G,0}^{T}$. But we can upgrade the definition for the case when all $q_e=0$.

Let  $\Phi_{G,0}^{In}$ be the graded commutative algebra over $\bK$ 
generated by the variables $\phi_e^{In}, e \in G$, with the defining relations:
$$(\phi_e^{In})^2 = 0, \quad \text {for every edge}\; e\in G$$
and 
 \begin{multline*} \left(\frac{ f_I^{In}-f_{\bar{I}}^{In}  }{q}\right)\Big{|}_{q=0}=0,
 \text{\ for every subset $g\in I\subseteq V(G)$}.
 \end{multline*}

Let $\C_G^{In}$ be the subalgebra of  $\Phi_{G,0}^{In}$ generated by the elements
$$X_i^{In}=\sum_{e: \ i\in e} c_{i,e}  \phi_e^{In}, \ i\in[n],$$
where $c_{i,e}$ are the same as in~\eqref{eq:def_cie}.
\begin{theorem}
\label{thm:Internal}
 For any graph $G$, the algebras  ${\B}_{G}^{In}$ and ${\C}_G^{In}$ are isomorphic,
 their total dimension over $\bK$ is equal to $T_G(0,1)$ and their Hilbert series is given by $$\HH  _{{\C}_G^{In}}(t)=T_G\left(0,\frac{1}{t}\right)\cdot t^{e(G)-v(G)+c(G)}.$$
\end{theorem}

For Hecke deformations, we have two similar theorems. Their proofs are analogous to Theorem~\ref{thm:same} and to Theorem~\ref{thm:quo-Hecke} resp.
\begin{theorem}
\label{thm:same-tree}
For any loopless connected graph $G$, the filtrations of  its Hecke deformation $\Psi_{G,q}^T$ ($\Psi_{G,q}^{In}$) induced by $X_i^T$ ($X_i^{In}$) and induced from the algebra $\Phi_{G,q}^T$ ($\Phi_{G,q}^{In}$)   coincide. Furthermore the Hilbert polynomial $\HH  _{\Psi_{G,q}^T}(t)$ and $\HH  _{\Psi_{G,q}^{In}}(t)$  of this filtration sre given by 
$$\HH  _{\Psi_{G,q}^T}(t)=T_G\left(1,\frac{1}{t}\right)\cdot t^{e(G)-v(G)+c(G)},$$
$$\HH  _{\Psi_{G,q}^{In}}(t)=T_G\left(0,\frac{1}{t}\right)\cdot t^{e(G)-v(G)+c(G)}.$$
\end{theorem}
\begin{theorem}
\label{thm:quo-Hecke-tree}
Let $G$ be a graph and $q\in \bbK$ ($q_e=q,\ \forall e\in G$). Then all defining relations between $X_i^*,\ i\in [n]$ ($*=T$ or $In$) are given by
$$\prod_{k=-\vec{d}_I}^{\cev{d}_I-1}\left(\sum_{i\in I} X_i^T -qk\right)=0$$
and
$$\prod_{k=-\vec{d}_I+1}^{\cev{d}_I-1}\left(\sum_{i\in I} X_i^{In} -qk\right)=0,$$
where $I$ is any subset of vertices and $\vec{d}_I$ (resp. $\cev{d}_I$) is the number of edges $e=(i,j)\in G:\ i\in I,\ j\notin I$ and $i>j$ (resp. $i<j$).  
\end{theorem}

\bigskip

In general case we should calculate root-connected and strong-connected score vectors instead  all score vectors.
\begin{definition}
Orientation $\vec{G}$ is  called a $g$-connected (strong-connected) orientation if from  any vertex there is a path to $g$ (and from $g$ to it). The corresponding score vector $D_{\vec{G}}^+$ is called a $g$-connected (strong-connected) score vector. \end{definition}
\begin{theorem}
\label{thm:different-tree}
 For any loopless weighted connected graph $G$ with a root $g$, the dimensions of the algebras $\Psi_G^T$ and $\Psi_G^{In}$ are equal to the number of distinct $g$-connected score vectors and to the number of distinct strong-connected score vectors resp.
\end{theorem}
\begin{proof}
The proof of Theorem~\ref{thm:different-tree} is more complicated than of Theorem~\ref{thm:different}, the key idea is that $\Psi_{G}^T$  and  $\Psi_{G}^{In}$ are quotient algebras of $\Psi_{G}$. We will prove it only for the case of $\Psi_{G}^T$, the second case is the similar. Namely, $$\Psi_{G}^T=\Psi_{G}/{\mathcal P^T},$$
where ${\mathcal P^T}\subset \Phi_G$ is  an ideal generated by $f_I,$ for $g\ni I\subseteq V(G)$.

Consider an element $f_I$ as expressed in  tilde-coordinates

\begin{equation}
[f_I]_{E'}=\left\{\begin{gathered}
1,\text{\ if $E'\cap (E(I,\bar{I})\cup E(\bar{I},I))=E(I,\bar{I})$}\\
0,\text{\ otherwise}
\end{gathered}\right.
\end{equation}
Since $\mathcal{P}^T$ is an ideal, any element $[0,\ldots,0,1,0,\ldots, 0]$ for which $1$ corresponds to non $g$-connected orientation $\vec{G}$ belongs to this ideal. It means that we can forget about coordinates such that corresponds to non $g$-connected $\vec{G}$ (in internal case, strong-connected).
\end{proof}

\begin{proof}[proof of Theorems~\ref{thm:Internal},~\ref{thm:same-tree} and  \ref{thm:quo-Hecke-tree}]
Here we present the proof only for the Internal case (for the case of trees, we already have an analogue of Theorem~\ref{thm:Internal}).  It is well know that for usual graphs the number of distinct strong-connected (root-connected) vectors is equal to $T(0,1)$ ($T(1,1,)$, i.e., the number of trees).

Similar to Theorem~\ref{thm:quo-Hecke}, the algebra $\Psi_{G,q}^{In}$ has relations from Theorem~\ref{thm:quo-Hecke-tree}. Let $\Psi_{G,q}'^{In}$ be the algebra which has only this relations, then we have $$\Dim(\B_G^{In})=\Dim(\Psi_{G,0}'^{In})\geq \Dim(\Psi_{G,q}'^{In})\geq \Dim(\Psi_{G,q}^{In}) = T(0,1) = \Dim(\B_G^{In}).$$
Now we proved Theorems~\ref{thm:same-tree} and  \ref{thm:quo-Hecke-tree}.
It remains to prove Theorem~\ref{thm:quo-Hecke}. From one side we know that 
$$\Dim(\C_G^{In})\geq \Dim(\Psi_{G,q}^{In}),$$ because we consider only maximal degrees inside each relation.
From another side we can check that $\C_G^{In}$ is a quotient algebra of $\B_G^{in}$.
Really, check the relation for a subset $I\subset V(G)$
$$\left(\sum_{i\in I} X_I\right)^{d_I-1}=\left(\sum_{e\in E(I,\bar{I})}\phi_e^{In}-\sum_{e\in E(\bar{I},I)}\phi_e^{In}\right)^{d_I-1}=f_I^{In}$$

Which gives that
$$T(0,1)=\Dim(\B_G^{In})\geq\Dim(\C_G^{In})\geq \Dim(\Psi_{G,q}^{In})=T(0,1),$$
then the dimensions of $\C_G^{In}$ and of $\B_G^{In}$ are the same, hence, algebras are isomorphic. 
\end{proof}

\begin{remark}
Note that in Theorem~\ref{thm:different-tree} (unlike Theorem~\ref{thm:different}) it is not true that if we change signs of some $q_e$, the dimension remains the same. Also we do not have combinatorial analogue of Theorem~\ref{thm:split}. In case of Hecke deformation it is still true (see Theorem~\ref{thm:same-tree}).
\end{remark}

\begin{problem}
Let $G$ be a connected graph with a root $g$,  given a partition $E=E_1\sqcup\ldots\sqcup E_k$ of edges and generic $q_1,\ldots,q_k\in \bK$ ($q_e=q_i$, for $e\in E_i$).
Describe the dimension of the algebra $\Psi_{G,Q}^T$ in terms of trees and forests.
\end{problem}


\bigskip

\noindent{\bf Acknowledgements} The authors want to thank Boris Shapiro for  useful discussions. The first author is grateful to the department of mathematics at Stockholm University for the hospitality in October 2016 when this project was carried out.


\begin{thebibliography}{8}
\bibitem{AP} F.~Ardila, A.~Postnikov, 
{\it Combinatorics and geometry of power ideals,} Trans. of the AMS 362:8 (2010), pp 4357--4384.


\bibitem{Ar} V.I.\,Arnold,
{\it Remarks on eigenvalues and eigenvectors of {H}ermitian matrices, {B}erry phase, adiabatic connections and quantum {H}all effect}, Selecta Mathematica 1:1 (1995), 1--19
	
	
\bibitem{B}
Berget A., Products of linear forms and {T}utte polynomials,
{\textit{European~J. Combin.}} \textbf{31} (2010), 1924--1935.

\bibitem{Be} O.~Bernardi, {\it Tutte polynomial, subgraphs, orientations and sandpile model: new connections via embeddings,} Electronic J. Combinatorics 15:1 (2008).


 
\bibitem{HR} O.~Holtz, A.~Ron, {\it Zonotopal algebra,} Advances in Mathematics 227 (2011), pp 847--894.


\bibitem{KW} D.J.~Kleitman, K.J.~Winston, {\it Forests and score vectors,} Combinatorica 1:1 (1981), pp 49--54.

\bibitem{NS} G.~Nenashev, B.~Shapiro, {\it "K-theoretic" analog of Postnikov-Shapiro algebra distinguishes graphs,} Journal of Combinatorial Theory, Series A, 148 (2017), pp. 316--332.

\bibitem{Ne} G.~Nenashev, {\it Postnikov-Shapiro Algebras, Graphical Matroids and their generalizations,}   \url{https://arxiv.org/abs/1509.08736}

\bibitem{OTe}
Orlik P., Terao H., Commutative algebras for arrangements, \textit{Nagoya
 Math.~J.} \textbf{134} (1994), 65--73.


\bibitem {PS} A.~Postnikov, B.~Shapiro, {\it Trees, parking functions, syzygies, and deformations of monomial ideals,} Trans. Amer. Math. Soc. 356:8 (2004),  pp 3109--3142.

\bibitem {PSS} A.~Postnikov,  B.~Shapiro, M.~Shapiro, {\it Algebras of curvature forms on homogeneous manifolds,} Differential topology, infinite-dimensional Lie algebras, and applications, Amer. Math. Soc. Transl. Ser. 2, 194 (1999), pp 227--235.

\bibitem {SS} B.~Shapiro, M.~Shapiro,  {\it On ring generated by Chern $2$-forms on ${\rm SL}_n/B$,} C. R. Acad. Sci. Paris. I Math. 326:1 (1998),  pp 75--80.

\end{thebibliography}
%

\end{document}